\documentclass[11pt]{article}
\usepackage[latin9]{inputenc}
\usepackage[active]{srcltx}
\usepackage{amsmath}
\usepackage{amsthm}
\usepackage{amssymb}
\usepackage{graphicx}
\usepackage{setspace}
\makeatletter

\usepackage{amscd}\usepackage{amsfonts}\textwidth=16.0cm \textheight=9.0in \topmargin=0.mm
\date{ }

\newcommand{\ga}{\Gamma}

\newtheorem{theorem}{Theorem}[section]
\newtheorem{lemma}[theorem]{Lemma}\newtheorem{corollary}[theorem]{Corollary}
\title{\bf Characterization of some alternating groups by order and the largest element order}
\smallskip
\normalsize
\author{{\bf Ali Mahmoudifar  \&  Ayoub Gharibkhajeh}
\\ Department of Mathematics, Tehran North Branch, Islamic Azad University, Tehran, Iran \\ 
e-mails:
a\_mahmoodifar@iau-tnb.ac.ir~~alimahmoudifar@gmail.com\\a.gharibkhajeh@gmail.com }

\makeatother

\begin{document}
\maketitle 
\begin{abstract}
The prime graph (or Gruenberg-Kegel
graph) of a finite group $G$ is a familiar graph. In this paper first, we investigate the structure of the finite groups with a non-complete prime graph. Then we prove that every alternating group $A_{n}$, where $n\leq20$ or $n\in\{23,24\}$ is determined by its order and its largest element order.
\end{abstract}
\textbf{2000 AMS Subject Classification}: $20$D$05$, $20$D$60$,
20D08. \\
 \textbf{Keywords :} Finite simple group, prime graph, the largest element orders, alternating group.

\section{Introduction}

Throughout this paper, $G$ denotes a finite group. The set of all prime divisors of $|G|$ is denoted by $\pi(G)$. Also, the set of all element orders of $G$ called the \textit{Spectrum} of $G$ is denoted by $\pi_{e}(G)$. The \textit{prime graph} (or \textit{Gruenberg-Kegel
graph}) of $G$ which is denoted
by $\ga(G)$ is a simple graph whose vertex set is $\pi(G)$ and two
distinct primes $p$ and $q$ are adjacent in $\ga(G)$ if and only
if $pq\in\pi_{e}(G)$. A subset $\rho$ of vertices of $\ga(G)$
is called an independent subset of $\ga(G)$, whenever every two distinct
primes in $\rho$ are non-adjacent in $\ga(G)$.

Let  $m_{1}(G)$ be the largest element
order of $G$, i.e. $m_{1}(G)$ is the maximum of $\pi_{e}(G)$. In general, if $k=|\pi_e(G)|$, then for $2\leq i\leq k$, we define $m_i(G)$ as follows:
\begin{equation*}
m_i(G)=\max\{ a ~|~ a\in\pi_e(G)\setminus\{m_1(G),\dots , m_{i-1}(G)\}\}
\end{equation*}

For a finite simple group $S$ there are a lot of results about the numbers $m_1(S)$, $m_2(S)$ and $m_3(S)$ (see \cite{Guest,Kantor}). Also, the characterization of finite simple groups by their arithmetical properties has been researched widely. For instance, Mazurov et al. in \cite{prov-conj}, show that every finite simple group $S$ can be determined by $|S|$ and $\pi_e(S)$. Then some authors tried to investigate the characterization of finite simple groups by using fewer conditions. In \cite{italia,comm}, it is proved that there some finite simple groups $S$ which are determined by $|S|$ and $m_1(S)$. For more results see \cite{jas,suz,L2q,K4,shi}.

In this paper first, we consider the finite groups whose prime graphs are not complete. Then as an application we prove the following theorem:

\begin{theorem}\label{Main 2}
	Let $G$ be a finite group and $A_{n}$ be an alternating group such that $n\leq20$ or $n\in\{23,24\}$. Then $G\cong A_n$ if and only if $|G|=|A_n|$ and $m_1(G)=m_1(A_n)$.
	\end{theorem} 
We note that our main tool for considering Theorem \ref{Main 2} is the fact that when $n\leq20$ or $n\in\{23,24\}$, the hypothesis in $m_1(G)=m_1(A_n)$ implies that $\Gamma(G)$ is not complete (see Table 1). However, if $n\in\{21, 22\}$ or $n\geq25$, then the condition $m_1(G)=m_1(A_{n})$ do not straightly show that $\Gamma(G)$ is a non-complete graph (see Lemma \ref{Lem4.1}). For instance if $n=21$, then $m_1(G)=m_1(A_{21})=420>19\cdot17$ which do not necessarily cause $19$ and $17$ are adjacent or even non-adjacent in $\Gamma(G)$. We note that
in the appendix, there is a procedure in Maple software for computing $m_1(A_n)$. 

Recall that ${\rm Soc}(G)$ denotes the socle of $G$ (the subgroup generated by all the
minimal nontrivial normal subgroups of $G$). The other notation and terminologies in this paper are standard and the reader is referred to \cite{key-13,key-12} if necessary.

\section{Preliminary Results}

\begin{lemma}\cite[Lemma 4]{key-6}\label{sec:Preliminary-Results-A_n and S_n}
In $S_{m}$ (resp. in $A_{m}$) there is an element of order $n=p_{1}^{\alpha_{1}}\,p_{2}^{\alpha_{2}}\,\cdots\,p_{s}^{\alpha_{s}}$,
where $p_{1},p_{2},\ldots,p_{s}$ are distinct primes and $\alpha_{1},\alpha_{2},\ldots,\alpha_{s}$
are naturals, if and only if $p_{1}^{\alpha_{1}}+p_{2}^{\alpha_{2}}+\cdots+p_{s}^{\alpha_{s}}\leq m$
(resp. $p_{1}^{\alpha_{1}}+p_{2}^{\alpha_{2}}+\cdots+p_{s}^{\alpha_{s}}\leq m$
for odd $n$ and $p_{1}^{\alpha_{1}}+p_{2}^{\alpha_{2}}+\cdots+p_{s}^{\alpha_{s}}\leq m-2$
for even $n$). \end{lemma}

\begin{lemma}\label{pl1}\cite[Lemma 1]{key-7} Let a finite group
$G$ have a normal series of subgroups $1\leq K\leq M\leq G$, and
the primes $p$, $q$ and $r$ are such that $p$ divides $|K|$,
$q$ divides $|M/K|$, and $r$ divides $|G/M|$. Then $p$, $q$,
and $r$ cannot be pairwise nonadjacent in $\ga(G)$. \end{lemma}
\begin{lemma}\label{Fro}(See, for example, \cite{key-12})
Let $G = F \rtimes H$ be a Frobenius group with
kernel $F$ and complement $H$. Then $|H|$ divides $|F| - 1$.
\end{lemma}
\begin{corollary}\label{lem natalia} Let $G$ be a finite group and
$N$ be a normal subgroup of $G$. Then the following assertions hold:

$1)$ Let $p$ and $q$ be two distinct primes in $\pi(G)$. If $p\in\pi(N)$, $q\in\pi(G/N)$
and $\{p, q\}$ is an independent subset of $\Gamma(G)$, then $q\mid(|N|_{p}-1)$.

$2)$ Let $p$, $q$ and $r$ be three pariwise distinct primes in $\pi(G)$. If $p\in\pi(N)$ and $\{q ,r\}\subseteq\pi(G/N)$
and $G/N$ is solvable, then $p$, $q$
and $r$ cannot be pairwise nonadjacent in $\ga(G)$.
\end{corollary} 

\begin{proof}
1) Let $P$ be a Sylow $p$-subgroup of $N$. By Frattini's argument, $G/N\cong N_{G}(P)/N_{N}(P)$. In view of the hypothesis, we conclude that $N_G(P)$
contains an element of order $q$. So $N_{G}(P)$ contains a subgroup isomorphic to the semidirect product
$P\rtimes Q$ where $Q$ is a cyclic subgroup of order $q$. On the
other hand, by the assumption, $G$ does not contain any element of order $pq$. Hence, $Q$ acts fixed point freely on $P$. Thus, $P\rtimes Q$ is a Frobenius group and so by Lemma \ref{Fro}, $q\mid(|P|-1)$ which implies that $q\mid(|N|_{p}-1)$. 

2) Put $\bar{G}=G/N$ and $\rho=\{q, r\}$. Recall that $\bar{G}$ is a solvable group and $\rho\subseteq\pi(\bar{G})$. Take a Hall $\rho$- subgroup $\bar{H}$ of $\bar{G}$. We know that $O_q(\bar{H})\neq 1$ or $O_r(\bar{H})\neq 1$. So without loss of generality, we may assume that $\bar{G}$ contains a subgroup isomorphic to the semidirect product $\bar{H_1}\rtimes \bar{H_2}$ in which $\pi(\bar{H_1})=\{q\}$ and $\pi(\bar{H_2})=\{r\}$.

Now let $P$ be a Sylow $p$-subgroup of $N$. Similar to the previous case, it follows that $\bar{G}\cong N_{G}(P)/N_{N}(P)$. Recall that $\bar{H_1}\rtimes \bar{H_2}$ is a subgroup of $\bar{G}$. Consequently, $N_{G}(P)/N_{N}(P)$ contains a subgroup isomorphic to $\bar{H_1}\rtimes \bar{H_2}$. Hence, there is a normal series $1<N_{N}(P)<T_1<T_2$ in $N_G(P)$ such that $T_1/N_{N}(P) \cong \bar{H_1}$ and $T_2/N_{N}(P) \cong \bar{H_1}\rtimes \bar{H_2}$. Also, by the above argument, $p\in \pi(N_N(P))$, $\pi(T_1/N_{N}(P))=\{q\}$ and $\pi(T_2/T_1)=\{r\}$. Therefore, by Lemma \ref{pl1}, we get that the subset $\{p, q, r\}$ can not be an independent subset of $\Gamma(G)$, which completes the proof.
\end{proof}

\begin{lemma} \label{lem:S<G/M}Let $G$ be a finite group, $M$ be
	a normal subgroup of $G$ and $G/M$ contain a subgroup $S$ which is isomorphic to a simple
	group. If $R$ is a Sylow $r$-subgroup of $M$, then one of the following
	assertions holds:
	
	$1)$ $|S|\mid|{\rm Aut}(R)|$
	
	$2)$ If $a\in\pi_{e}(S)$ and $r^{\alpha}\in\pi_{e}(R)$, then ${\rm lcm}(r^{\alpha},a)\in\pi_{e}(G)$.
\end{lemma}
\begin{proof}
	Put $N=N_{G}(R)$, $L=N_{M}(R)$
	and $C=C_{G}(R)$. By Frattini's argument, $G/M\cong N/L$. Hence by the assumption we get that $N/L$ contains a subgroup isomorphic to the simple group $S$. Let $K$ be a subgroup of $N$ such that $K/L\cong S$ is a simple group. Since $K/L$ is a simple subgroup of
	$N/L$ and $CL/L$
	is a normal subgroup of $N/L$, it follows that either $K/L\cap CL/L=1$ or $K/L\leq CL/L$. We consider each possibiltiy:
	
	1) Let $K/L\cap CL/L=1$. Then we obtain the following relation:
\begin{equation*}
K/L\cong\frac{(K/L)(CL/L)}{CL/L}\leq\frac{N/L}{CL/L}\cong N/CL.
\end{equation*}

So $|K/L|\mid |N/CL|$. On the other hand:
\begin{equation*}
	|N/CL|=|\frac{N/C}{CL/C}|\mid|{\rm Aut}(R)|.
\end{equation*}

	Therefore, $|K/L|\mid|{\rm Aut}(R)|$ and consequently, $|S|\mid|{\rm Aut}(R)|$. 
	
	2) Let $K/L\leq CL/L$. Since $CL/L\cong C/C_{L}(R)$, it follows that $C/C_{L}(R)$ contains a subgroup isomorphic to $K/L$. Recall that $C=C_G(R)$.
	Hence if $a\in\pi_{e}(S)=\pi_e(K/L)$ and $r^{\alpha}\in\pi_{e}(R)$, then $a\in\pi_{e}(C_{G}(R))$
	and so ${\rm lcm}(r^{\alpha},a)\in\pi_{e}(G)$ which completes the proof. 
\end{proof}

\section{Groups with a non-complete prime graph}

\begin{lemma}\label{l1} Let $G$ be a finite group, $K_{1}$ and
$K_{2}$ two normal subgroups of $G$ and $\rho$ an independent
subset of $\ga(G)$. Then either $\pi(K_{1})\cap\rho\subseteq\pi(K_{2})\cap\rho$
or $\pi(K_{2})\cap\rho\subseteq\pi(K_{1})\cap\rho$. Moreover, if
$N$ is the product of all normal subgroups $K$ of $G$ such that
$|\pi(K)\cap\rho|\leq1$, then $|\pi(N)\cap\rho|\leq1$.
\end{lemma} 
\begin{proof}
For $1\leq i\leq2$, put $\pi_{i}=\pi(K_{i})\cap\rho$. If $\pi_1\nsubseteq\pi_2$ and $\pi_2\nsubseteq\pi_1$, then there exist two primes $p_{1}$ and $p_{2}$ such that
$p_{1}\in\pi_{1}\setminus\pi_{2}$ and $p_{2}\in\pi_{2}\setminus\pi_{1}$. This implies that $p_{1}\in\pi(K_{1}/(K_{1}\cap K_{2}))$ and $p_{2}\in\pi(K_{2}/(K_{1}\cap K_{2}))$.
By the following relation: 
\[
\frac{K_{1}K_{2}}{K_{1}\cap K_{2}}\cong\frac{K_{1}}{K_{1}\cap K_{2}}\times\frac{K_{2}}{K_{1}\cap K_{2}}
\]
it follows that $K_{1}K_{2}$ contains an element of order $p_{1}p_{2}$,
which contradices to the assumption. Therefore, $\pi_{1}\subseteq\pi_{2}$
or $\pi_{2}\subseteq\pi_{1}$ and consequently, there is $i\in\{1, 2\}$, such that $\pi(K_{1}K_{2})\cap\rho\subseteq \pi_i$. 
Also this implies that if $|\pi_1| \leq 1$ and $|\pi_2| \leq 1$, then $|\pi(K_{1}K_{2})\cap\rho|\leq 1$.

Finnaly, let $N$ be the product of all normal subgroups $K$ of $G$ such that
$|\pi(K)\cap\rho|\leq1$. Then by the above discussion, $|\pi(N)\cap\rho|\leq1$, which completes the proof. 
\end{proof}
We note that by the previous lemma, if $\rho$ is an independent subset of $\ga(G)$ such that $|\rho|\geq 2$, then $G$ contains a normal subgroup $N$ which is the largest normal subgroup of $G$ among the normal subgroups of $G$ with the property $|\pi(N)\cap\rho|\leq 1$.
\begin{theorem}\label{main} Let $G$ be a finite group and $\rho$
be an independent subset of $\ga(G)$ such that $|\rho|\geq2$. Then
one of the following assertions holds:

$1)$ $G$ has a normal series $1\leq N\leq L\leq G$, where $L/N={\rm Soc}(G/N)$ is the socle of $G/N$. Moreover, in this case $\pi(N)\cap\rho=\{p\}$, $\pi(L/N)\cap\rho=\{q\}$ and $\rho=\{p, q\}$.

$2)$ There exists a normal subgroup $N$ of $G$ and a non-abelian
simple group $S$ such that 
\[
S\leq\frac{G}{N}\leq{\rm Aut}(S),
\]
where $|\pi(N)\cap\rho|\leq 1$ and $|\pi(S)\cap\rho|\geq2$. Moreover, if $|\rho|\geq 3$, then $|\pi(S)\cap\rho|\geq|\rho|-1$.
\end{theorem} 
\begin{proof}
Let $G$ be a finite group, $\rho$ be an independent subset
of $\ga(G)$ such that $|\rho|\geq 2$ and $N$ be the product of all normal subgroups $K$ of
$G$ such that $|\pi(K)\cap\rho|\leq1$. Also let $L/N$ be the socle
of $G/N$. By Lemma \ref{l1}, $|\pi(N)\cap\rho|\leq1$. Let $M_{1}/N,\ldots,M_{t}/N$ be
the minimal normal subgroups of $G/N$
such that $L/N\cong M_{1}/N\times\cdots\times M_{t}/N$. We know that for
each $1\leq i\leq t$, $M_{i}/N$ is a direct product of some isomorphic
simple groups. Also since  $N$ is a pure subgroup of $M_{i}$, $|\pi(M_{i})\cap\rho|>1$ and so
$|\pi(M_{i}/N)\cap\rho|\geq1$. In the sequel, we consider the following cases, seperaitly:

1) Let for every $1\leq i\leq t$, $|\pi(M_{i}/N)\cap\rho|=1$. In view of the definition of $N$, we conclude that there exist two distinct primes $p$ and $q$ such that $\pi(N)\cap \rho=\{p\}$ and for every $1\leq i\leq t$, $\pi(M_i/N)\cap \rho=\{q\}$. This implies that $\pi(L/N)\cap\rho=\{q\}$. 

By the above discussion, $\{p, q\}\subseteq\rho$. Let there exist $r\in\rho\setminus\{p, q\}$. Recall that, $\pi(N)\cap \rho=\{p\}$ and $\pi(L/N)\cap\rho=\{q\}$. This shows that $r\in\pi(G/L)$. On the other hand, $\{p, q, r\}$ is an independent subset of $\Gamma(G)$, which contradictes to Lemma \ref{pl1}. Therefore, $\rho=\{p, q\}$, which get the assertion (1) in the theroem.

2) Let there exist $1\leq i\leq t$, such that $|\pi(M_{i}/N)\cap\rho|\geq2$. Without lose of generality, suppose that $|\pi(M_{1}/N)\cap\rho|\geq2$.
In this case, if $t\geq2$, then $M_{1}/N\times M_{2}/N$ contains
an element of order $pq$ where $p\in\pi(M_{1}/N)\cap\rho$ and $q\in\pi(M_{2}/N)\cap\rho$,
which is a contradiction. Thus, $t=1$. Also since $M_{1}/N$ is
a direct product of some isomorphic simple groups, by a similar argument, we conclude that $L/N=M_{1}/N$ is isomorphic to a non-abelian
simple group. Then in this case, $C_{G/N}(L/N)=1$ since $L/N$ is the socle of $G/N$. Let $L/N$ be isomorphic to a non-abelian simple group $S$. So the following relation holds:
\[
S\leq\bar{G}:=\frac{G}{N}\leq{\rm Aut}(S).
\]
We recall that in this case, $L/N=M_1/N\cong S$ and by the assumption $|\pi(M_{1}/N)\cap\rho|\geq2$. So $|\pi(S)\cap\rho|\geq 2$. 

Finnaly, we prove that if $|\rho|\geq 3$, then $|\pi(S)\cap\rho|\geq|\rho|-1$. On the
contrary, let $|\rho|\geq 3$ and $|\pi(S)\cap\rho|\leq|\rho|-2$. This implies that there are two distinct primes $p$ and $q$ in $\rho$ such that $\{p, q\}\subseteq\pi(N)\cup\pi(\bar{G}/S)$ and $\{p, q\}\cap\pi(S)=\emptyset$. Since $|\rho|\geq 3$, if $\{p,q\}\subseteq\pi(\bar{G}/S)$, then by Corollary \ref{lem natalia} (Assertion 2), we get a contradiction since $\bar{G}/S$ is solvable. Similarly, if $p\in\pi(N)$ and $q\in\pi(\bar{G}/S)$, then by Lemma \ref{pl1}, we arraive a contradiction. Therefore, when $|\rho|\geq 3$, we deduce that $|\pi(S)\cap\rho|\geq|\rho|-1$ which completes the proof.
\end{proof}
\textbf{Example.} Let $G=11^{2}:SL_{2}(5)$ , which is a Frobenius
group with kernel $11^{2}$ and complement $SL_{2}(5)$. In the prime
graph of $G$, the subsets $\rho_{1}=\{2,11\}$ and $\rho_{2}=\{11,3,5\}$
are two independent subsets. If we choose $\rho_{1}$ as the independent subset said in Theorem \ref{main}, then we have $N=11^{2}$ and $L=11^2 : 2$ which shows that Case (1) of Theorem \ref{main} holds. Also if we choose $\rho_{2}$ as the independent subset $\rho$ in Theorem \ref{main}, then $N=11^{2} : 2$ and we have 
\[
PSL_{2}(5)\leq G/N\leq{\rm Aut}(PSL_{2}(5)),
\]
which satisfies Case (2) of Theorem \ref{main}.

Now by Theorem \ref{main}, we can easily get the following two corollaries
which modify \cite[Lemma 10]{key-14} and \cite[Lemma 2.3]{K4}:

\begin{corollary}

If $G$ is a finite group and $\rho$ an independent subset of $\ga(G)$
such that $|\rho|\geq3$, then there exists a nonabelian simple group
$S$ and a normal subgroup $N$ of $G$ such that 
\[
S\leq\frac{G}{N}\leq{\rm Aut}(S),
\]
and also we have $|\pi(S)\cap\rho|\geq|\rho|-1$ and $|\pi(N)\cap\rho|\leq1$.

\end{corollary}

\begin{corollary}\label{coro-main}

Let $G$ be a finite group, $\rho$ be an independent subset of $\Gamma(G)$
such that $|\rho|\geq2$. Also let for every distinct prime numbers $p$
and $q$ belong to $\rho$ we have $p \nmid (q^j - 1)$ and $q \nmid (p^i - 1)$ where 
$1 < p^i \leq |G|_p$ and $1 < q^j \leq |G|_q$. Then there exists
a non-abelian simple group $S$ such that 
\[
S\leq\frac{G}{O_{\rho'}(G)}\leq{\rm Aut}(S),
\]
and also we have $\rho\subseteq\pi(S)$ and $\rho\cap\pi({\rm Out}(S))=\emptyset$.

\end{corollary}
\begin{proof}
It immediately comes from Theorem \ref{main} and Corollary \ref{lem natalia}.
\end{proof}
\section{Proof of Theorem \ref{Main 2}}
Recall that in number theory $Landau(n)$ is a familar notation for $m_1(S_n)$.
\begin{lemma}\label{Lem4.1} Let $A_{n}$ be an alternating group. If $n\geq25$
	or $n\in\{21,22\}$, then $m_{1}(A_n)\geq pq$ for all distinct prime $p$ and $q$ in $\pi(A_n)$.
\end{lemma} 
\begin{proof}
Let $p$ and $q$ be two distincet primes in $\pi(A_n)$. By the definition of $m_{1}(A_{n})$ and Lemma \ref{sec:Preliminary-Results-A_n and S_n}, $m_{1}(A_{n})\geq m_{1}(S_{n-2})=Landau(n-2)$. In view of \cite{key-15}, if $n\geq906$, then
\begin{equation*}
Landau(n)\geq e^{\sqrt{n\,ln(n)}}.
\end{equation*}
Hence,
\[
m_{1}(A_{n})\geq  e^{\sqrt{(n-2)\,ln(n-2)}}.
\]
On the other hand, by the hypothesis, $(n-2)^{3}>n\,(n-2)\geq p\,q$. Using an easy computation, we can show that if $n\geq906$, then
\[
 e^{\sqrt{(n-2)\,ln(n-2)}}\geq(n-2)^{3},
\]
Thus, by the above argument if $n\geq906$, then $m_{1}(A_{n})\geq(n-2)^{3}$ and consequently, $m_{1}(A_{n})>pq$. Finnaly, by the program in the appendix, and an easy compution we deduce that if $25\leq n\leq905$
or $n\in\{21,22\}$, then $m_{1}(A_{n})\geq p\,q$ which completes the proof. 
\end{proof}
{\bf Proof of Theorem \ref{Main 2}.}
Let $G$ be a finite group such that $|G|=|A_{n}|$ and $m_{1}(G)=m_{1}(A_{n})$ where $n\leq20$ or $n\in\{23,24\}$. If $n\leq 4$, then it is obvious that $G\cong A_n$. Let $n\geq 5$. By \cite[Theorem 1]{comm}, if $n\in\{5, 6\}$ , then $G\cong A_n$.  

Let $7\leq n\leq 20$ or $n\in\{23,24\}$. By Table 1, there exists an independent subset $\rho$ of $\Gamma(G)$ such that $\rho$ satisfies the conditions of Corollay \ref{coro-main}, which implies that there is a non-abelian simple group $S$ such that
\[
S\leq\bar{G}:=\frac{G}{M}\leq {\rm Aut}(S)
\]
where $M=O_{\rho'}(G)$, $\rho\subseteq \pi(S)$ and $\pi(\bar{G}/S)\cap \rho=\emptyset$. Moreover, by the assumption $|S|\mid |A_n|$ and . In view of \cite[Table 1]{key-9}, the possible cases for $S$ are indicated in Table 1. Hence, if $n\in\{7,13,14,17,19,23\}$, then by Table 1, $S\cong A_{n}$ and so $G\cong A_n$ since $|G|=|A_n|$. In the sequel, for the other cases, suppose that $S$ is not isomorphic to $A_n$.

Let $n=8$. By Table l, $S\cong A_{7}$ or $L_{3}(4)$. If $S\cong A_{7}$,
then $G/M$ is isomorphic to either $A_{7}$ or $S_{7}$ and $|M|\mid8$.
On the other hand, $A_{7}$ and $S_{7}$ do not contain any element
of order $15$, in while $m_{1}(G)=m_{1}(A_{8})=15$, which is a contradiction.
If $S\cong L_{3}(4)$, then $|S|=|A_{8}|$ and so $G\cong L_{3}(4)$
which is impossible since by \cite{key-13}, $m_{1}(L_{3}(4))=7$.

Let $n=9$. By Table l, $S\cong A_{8},A_{7}$ or $L_{3}(4)$.
If $S\cong A_{8},A_{7}$ or $L_{3}(4)$, then $7\mid |G/M|$ and $|M|_{3}=3$ or $9$.
By Corollary \ref{lem natalia}, we get that $G$ contains an element
of order $21$ which is a contradiction since $m_{1}(G)=m_{1}(A_{9})=15$.

Let $n=10$. By Table l, $S\cong J_{2}$. Then by \cite{key-13}, we deduce that $|M|=9$ and $S$ contains an element of order $10$. Hence by Lemma \ref{lem:S<G/M}, we get that $G$ contains
an element of order $30$ which is a contradiction since $m_{1}(G)=m_{1}(A_{10})=21$.

Let $n=11$. By Table l, $S\cong M_{22}$. Then $11\mid |S|$ and $|M|_{3}=3^{2}$. So by Corollary \ref{lem natalia}, we get that $G$ contains an element of order $33$
which is a contradiction since $m_{1}(G)=m_{1}(A_{11})=21$. 

Let $n=12$. By Table l, $S\cong A_{11}$ or $M_{22}$. Let
$S\cong M_{22}$. Then $11\mid |S|$ and $|M|_5=5$. So by Lemma \ref{lem:S<G/M}, we get that $55\in\pi_e(G)$ which is impossible since $m_1(G)=m_1(A_{12})=35$. Let $S\cong A_{11}$. Then $|M|_{3}=3$ and $S$ contains an element of order $20$. So by Lemma \ref{lem:S<G/M}, we get that $60\in\pi_e(G)$ which is a contradiction.

Let $n=14$. By Table l, $S\cong A_{13}$. Then $|M|_{7}=7$. So $|M|_7=7$ and $S$ contains an
element of order $30$. Hence by Lemma \ref{lem:S<G/M}, we get that
$210\in\pi_{e}(G)$ which is a contradiction since $m_{1}(G)=m_{1}(A_{14})=60$.

Let $n=15$. By Table l, $S\cong A_{14}$ or $A_{13}$. Let
$S\cong A_{13}$ or $A_{14}$. Then $|M|_{5}=5$ and $S$ contains an
element of order $28$. Hence by Lemma \ref{lem:S<G/M}, we get that
$140\in\pi_{e}(G)$ which is a contradiction since $m_{1}(G)=m_{1}(A_{15})=105$.

Let $n=16$. By Table l, $S\cong A_{15},A_{14}$ or $A_{13}$.
We note that $m_{1}(G)=m_{1}(A_{16})=105$. So if $S\cong A_{13}$
or $A_{14}$, then similar to the case $n=15$, we get that $140\in\pi_{e}(G)$ which is a contradiction.
Let $S\cong A_{15}$. In this case, we have $S$ contains an element
of order $105$ and also $|M|=8$ or $16$. Thus, by Lemma \ref{lem:S<G/M},
we get that $210\in\pi_{e}(G)$ which is impossible. 

Let $n=18$. By Table l, $S\cong A_{17}$ and $m_{1}(G)=m_{1}(A_{18})=140$.
If $S\cong A_{17}$, then $|M|_{3}=9$ and $70\in\pi_e(S)$. So by Lemma \ref{lem:S<G/M},
$210\in\pi_{e}(G)$, which is impossible. 

Let $n=20$. By Table l, $S\cong A_{19}$ and $m_{1}(G)=m_{1}(A_{20})=210$.
If $S\cong A_{19}$, then $|M|_{5}=5$ and $77\in\pi_{e}(S)$, and so by Lemma \ref{lem:S<G/M}, $5\cdot77\in\pi_{e}(G)$,
which is impossible. 

Let $n=24$. By Table l, $S\cong A_{23}$ and $m_{1}(G)=m_{1}(A_{24})=420$.
If $S\cong A_{23}$, then $|M|_{3}=3$ and $385\in\pi_{e}(A_{23})$ and so by Lemma \ref{lem:S<G/M}, $3\cdot385\in\pi_{e}(G)$, which is impossible. 

Finally, by the above discussions we conclude that if $|G|=|A_n|$ and $m_1(G)=m_1(A_n)$, then $S\cong A_n$ and consequently, $G\cong A_n$ which completes the proof.
\section*{
\begin{table}
\protect\protect\protect\caption{\label{tab:1-1}$|S|\mid|A_{n}|$ and $\rho\subseteq\pi(S)$}
\protect\centering{}%
\begin{tabular}{|c|c|c|c|c|}
\hline 
$n$ & $|A_{n}|$ & $m_{1}(A_{n})$ & $\rho$ & $S$\tabularnewline
\hline 
$7$ & $2^{3}\cdot3^{2}\cdot5\cdot7$ & $7$ & $\{5,7\}$ & $A_{7}$\tabularnewline
\hline 
$8$ & $2^{6}\cdot3^{2}\cdot5\cdot7$ & $15$ & $\{5,7\}$ & $A_{8},L_{3}(4),A_{7}$\tabularnewline
\hline 
$9$ & $2^{6}\cdot3^{4}\cdot5\cdot7$ & $15$ & $\{5,7\}$ & $A_{9},A_{8},L_{3}(4),A_{7}$\tabularnewline
\hline 
$10$ & $2^{7}\cdot3^{4}\cdot5^{2}\cdot7$ & $21$ & $\{5,7\}$ & $A_{10},J_{2}$\tabularnewline
\hline 
$11$ & $2^{7}\cdot3^{4}\cdot5^{2}\cdot7\cdot11$ & $21$ & $\{7,11\}$ & $A_{11},M_{22}$\tabularnewline
\hline 
$12$ & $2^{9}\cdot3^{5}\cdot5^{2}\cdot7\cdot11$ & $35$ & $\{7,11\}$ & $A_{12},A_{11},M_{22}$\tabularnewline
\hline 
$13$ & $2^{9}\cdot3^{5}\cdot5^{2}\cdot7\cdot11\cdot13$ & $35$ & $\{7,11,13\}$ & $A_{13}$\tabularnewline
\hline 
$14$ & $2^{10}\cdot3^{5}\cdot5^{2}\cdot7^{2}\cdot11\cdot13$ & $60$ & $\{11,13\}$ & $A_{13},A_{14}$\tabularnewline
\hline 
$15$ & $2^{10}\cdot3^{6}\cdot5^{3}\cdot7^{2}\cdot11\cdot13$ & $105$ & $\{11,13\}$ & $A_{13},A_{14},A_{15}$\tabularnewline
\hline 
$16$ & $2^{14}\cdot3^{6}\cdot5^{3}\cdot7^{2}\cdot11\cdot13$ & $105$ & $\{11,13\}$ & $A_{13},A_{14},A_{15},A_{16}$\tabularnewline
\hline 
$17$ & $2^{14}\cdot3^{6}\cdot5^{3}\cdot7^{2}\cdot11\cdot13\cdot17$ & $105$ & $\{11,13,17\}$ & $A_{17}$\tabularnewline
\hline 
$18$ & $2^{15}\cdot3^{8}\cdot5^{3}\cdot7^{2}\cdot11\cdot13\cdot17$ & $140$ & $\{11,13,17\}$ & $A_{18},A_{17}$\tabularnewline
\hline 
$19$ & $2^{15}\cdot3^{8}\cdot5^{3}\cdot7^{2}\cdot11\cdot13\cdot17\cdot19$ & $210$ & $\{13,17,19\}$ & $A_{19}$\tabularnewline
\hline 
$20$ & $2^{17}\cdot3^{8}\cdot5^{4}\cdot7^{2}\cdot11\cdot13\cdot17\cdot19$ & $210$ & $\{13,17,19\}$ & $A_{19},A_{20}$\tabularnewline
\hline 
$23$ & $2^{18}\cdot3^{9}\cdot5^{4}\cdot7^{3}\cdot11^{2}\cdot13\cdot17\cdot19\cdot23$ & $420$ & $\{19,23\}$ & $A_{23}$\tabularnewline
\hline 
$24$ & $2^{21}\cdot3^{9}\cdot5^{4}\cdot7^{3}\cdot11^{2}\cdot13\cdot17\cdot19\cdot23$ & $420$ & $\{19,23\}$ & $A_{23},A_{24}$\tabularnewline
\hline 
\end{tabular}\protect
\end{table}
}

\part*{Appendix}

{\it with(NumberTheory):  with(ArrayTools):\\   m\_1Alt:=proc(n) local l,T\_o,T\_e,i,t,T;\\
 l:=proc(m) 
 local S,A,B,k,r ; \\ S:=0:A:=ifactors(m):B:=A[2]: k:=Size(B): \\
 for r from 1 to $\frac{k[2]}{2}$ do $S:=S+(B[r][1])^{B[r][2]}$ end do: S ; end proc:\\
$T\_o:=\{\}: T\_e:=\{\}$:  \\
for i from 1 to Landau(n) do if i::even and $l(i)\leq n-2$\\ then $T\_e:=T\_e ~union~ \{i\}$ end if ; end do: \\
for i from 1 to Landau(n) do if i::odd  and $l(i)\leq n$ \\
then $T\_o:=T\_o ~union~ \{i\}$ end if ; end do: \\ $T:=T\_o ~union~ T\_e$:~t:=max(T): \\
 print( $^\prime$ The order of Alt(n) $^\prime$=ifactor($\frac{Factorial(n)}{2}$)); \\
  print( $^\prime$ The spectrum of Alt(n) $^\prime$=T); \\
   print( $^\prime$ The largest element order of Alt(n) $^\prime$=t);  end proc:
}
\\

\end{document}